\documentclass[12pt]{amsart}

\usepackage{amsmath}
\usepackage{amsthm}
\usepackage{amssymb}
\usepackage{graphicx}
\usepackage{verbatim}
\usepackage{lineno}

\theoremstyle{plain}
\newtheorem{thm}{Theorem}[section]

\newtheorem{lem}[thm]{Lemma}
\newtheorem{prop}[thm]{Proposition}

\theoremstyle{definition}
\newtheorem{defn}{Definition}[section]
\newtheorem*{claim*}{Claim}

\theoremstyle{remark}
\newtheorem{rem}{Remark}
\newtheorem*{eg*}{Example}

\newcommand{\url}[1]{\texttt{#1}}

\newcommand{\R}{\mathbb{R}}
\newcommand{\Q}{\mathbb{Q}}
\newcommand{\C}{\mathbb{C}}
\newcommand{\F}{\mathbb{F}}
\newcommand{\N}{\mathbb{N}}
\newcommand{\Z}{\mathbb{Z}}
\newcommand{\T}{\mathbb{T}}

\newcommand{\cB}{\mathcal{B}}

\newcommand{\abs}[1]{\left\lvert #1 \right\rvert}
\newcommand{\norm}[1]{\left\lVert #1 \right\rVert}
\newcommand{\set}[2]{\left\{#1 \ : \ #2\right\}}

\DeclareMathOperator{\id}{id}
\DeclareMathOperator{\subspan}{span}

\newcommand\restr[2]{{
  \left.\kern-\nulldelimiterspace 
  #1 
  \vphantom{\big|} 
  \right|_{#2} 
  }}

\begin{document}

\title{On Irreducibility of Oseledets Subspaces}
\author{Christopher Bose}
\email{cbose@uvic.ca}
\author{Joseph Horan}
\email{jahoran@uvic.ca}
\author{Anthony Quas}
\email{aquas@uvic.ca}
\address{Department of Mathematics and Statistics, University of 
Victoria, PO BOX 3045 STN CSC, Victoria, B.C., V8W 3R4, Canada}
\thanks{All threee authors are supported by grants from the Natural 
Sciences and Engineering Research Council of Canada.}
\date{\today}

\begin{abstract}
For a cocycle of invertible real $n$-by-$n$ matrices, the Multiplicative 
Ergodic Theorem gives an Oseledets subspace decomposition of $\R^n$; that is, 
above each point in the base space, $\R^n$ is written as a 
direct sum of equivariant subspaces, one for each Lyapunov exponent of the cocycle. 
It is natural to ask if these summands may be further decomposed into equivariant subspaces;
that is, if the Oseledets subspaces are reducible. 
We prove a theorem yielding sufficient conditions for irreducibility of 
the trivial equivariant subspaces $\R^2$ and $\C^2$ for $O_2(\R)$-valued
cocycles and give explicit examples where the conditions are satisfied.
\end{abstract}

\maketitle{}

\vskip 1ex



\section{Introduction}
\label{sec:intro}

The Multiplicative Ergodic Theorem (MET) has a rich history of generalizations 
and variations; there are versions of the theorem in many different situations 
(see Froyland, Lloyd, Quas \cite{flq2} for a brief survey).
The original theorem, by Oseledets \cite{oseledets}, 
obtained a splitting of $\R^n$ into equivariant subspaces, each subspace 
corresponding to a different Lyapunov exponent of the differentiable matrix 
cocycle; these splittings (and generalizations thereof) are now called Oseledets splittings. 

These splittings play a very important role in the study of hyperbolic and 
non-uniformly hyperbolic dynamical systems, giving the tangent spaces for 
invariant sub-manifolds. Coming from differential equations, the exponential 
dichotomy, or Sacker-Sell spectrum, gives a splitting of $\R^n$ into equivariant 
subspaces with uniform gaps between exponential growth rates 
\cite{sacker_sell}. A natural question is that of reducibility of an 
equivariant family of subspaces: when can one find a lower-dimensional 
equivariant subspace inside a given one. One can check that the subspaces
appearing in the Sacker-Sell decomposition are a sum of subspaces appearing
in the Oseledets decomposition. It can happen that the (non-uniform) Oseledets decomposition
strictly refines the uniform Sacker-Sell decomposition 
\cite{john_palm_sell}, even in cases where the Sacker-Sell decomposition is 
trivial \cite{johnson_eg}. Bochi \cite{bochi_generic} has shown 
that generically for a matrix cocycle over a minimal base, the two decompositions 
coincide. Since these are results about reducibility of the Sacker-Sell decomposition, 
it is natural to ask further about reducibility of the Oseledets decomposition.

One way to look at Oseledets splittings, for the case of real 
invertible matrix cocycles, is to use the notion of \emph{block diagonalization}
of the cocycle. From the MET, we get a splitting of $\R^n$ which is equivariant; 
this equivariance property may be used to show that the cocycle is cohomologous 
to one which is block diagonal. The column vectors of the new cocycle exactly 
correspond to basis vectors of the subspaces in the splitting, and the blocks 
correspond to the subspaces. We then see that refinement of the Oseledets 
splitting, in the sense of further decomposing the subspaces, is equivalent 
to showing that the cocycle is cohomologous to a cocycle with a more 
refined block structure.

One of the key steps in the original proof of the Oseledets' theorem 
(\cite{oseledets}, in Russian, or \cite{pesinsmth}, by Barreira and Pesin) is 
to construct a related \emph{triangular} cocycle over an extended base space, 
thereby reducing the computation to the simpler case of triangular cocycles. 
Other parts of the theorem are obtained, however, using the original 
(non-triangularized) cocycle. A triangular cocycle is in some ways simpler 
than a non-triangularized cocycle, in that there is a flag of nested equivariant 
subspaces for the cocycle. In the two-dimensional case, a cocycle is irreducible 
if and only if it is not triangularizable.

More recently, Arnold, Cong, and Oseledets showed in \cite{ano_normal} that 
any real matrix cocycle over an invertible ergodic map, not just those satisfying 
the log-integrability condition of the MET, may be put into an equivariant form 
with blocks on the diagonals which are block-conformal, nothing below those
large blocks, and arbitrary elements above those blocks; they call this is a 
Jordan normal form. With another condition, the elements above the diagonal 
may be removed, so that the form is block-diagonal.
In the setting of $2$-by-$2$ real invertible matrix cocycles, Thieullen in \cite{thieullen} 
gave a classification of possible cocycles analogous to the classification of M\"obius 
transformations. 

In this paper, we consider $O_2(\R)$-valued cocycles, so that the only Lyapunov exponent is 0
and the decomposition arising from the Oseledets theorem is trivial. We exploit the
near-commutativity of $O_2$ to relate reducibility
of the trivial equivariant subspaces $\R^2$ and $\C^2$ to the ergodic properties of a pair of 
skew product extensions of the base dynamical system, and give sufficient conditions for 
irreducibility of the subspaces, both over the reals and over the complex numbers. 

The rationale for considering reducibility over the complex numbers is that even for 
a single matrix (the case where the underlying dynamics is a single fixed point), not
every matrix can be triangularized over the reals, while it can always be triangularized
over the complex numbers. 

Using our criteria, we build two simple examples of $O_2(\R)$-valued cocycles: one with the property that
the trivial complex equivariant subspace $\C^2$ is reducible, but the real subspace $\R^2$ is irreducible, 
and the other with the property that both $\C^2$ and $\R^2$ are irreducible.

We thank Kening Lu for bringing this problem to our attention.

\section{Preliminaries and Formulation}
\label{sec:prelims}

Throughout this section, we use $\F$ to refer to either $\R$ or $\C$. 
Any measurability requirements for matrix-valued functions are with 
respect to the Borel $\sigma$-algebra generated by the usual norm topology. 
Measurability of subspace-valued functions is with respect to the appropriate 
Grassmannian.

\begin{defn}[Invertible Matrix Cocycle]
Let $(X,\mathcal{B},\mu,T)$ be a dynamical system, where $T$ is an invertible 
measure-preserving transformation, and $A_0:X\to GL_d(\F)$ be measurable. 
Define the \emph{cocycle} $A: \Z\times X \to GL_d(\F)$ by: 
\begin{gather*}
A(1,x) = A_0(x), \\
A(0,x) = I, \\
A(n,x) = A_0(T^{n-1}(x)) \cdots A_0(T(x)) A_0(x), \\ 
A(-n,x) = A_0(T^{-n}(x))^{-1} \cdots A_0(T^{-1}(x))^{-1},
\end{gather*} 
for all $n\in\Z,\ n>0$. One easily checks that $A(n+m,x) = A(m,T^n(x)) A(n,x),$ 
for all $n,m\in\Z$. $A_0$ is the generator for the cocycle, and we often 
use the same letter for the cocycle and its generator.
\end{defn}

\begin{defn}[Equivariant Family of Subspaces]
Let $A$ be an invertible matrix cocycle over $(X,\cB,\mu,T)$. The measurable 
function $x \mapsto V(x)$ is called an \emph{equivariant family of subspaces} 
for A if there exists a $T$-invariant set of full measure $\tilde{X} \subset X$ 
such that for all $x\in \tilde{X}$, $A(1,x)V(x) = V(T(x))$.
\end{defn}

We note that an equivariant family of subspaces is equivalent to an invariant
measurable vector space bundle: the bundle $\{(x,V(x))\colon x\in \tilde X\}$
is invariant under $T_A(x,V)=(T(x),A(1,x)(V))$. 

Note that when we speak of invariant vector bundles, for the remainder of the 
paper, we are referring to \emph{measurable} invariant vector bundles.

\begin{defn}[Reducible Bundle]
Let $A$ be an invertible matrix cocycle, and let $V$ be an invariant measurable vector
bundle over $X$. We say that $V$ is \emph{reducible} if there exists an invariant 
vector bundle $W$ over $X$ such that for $\mu$-a.e.\ $x$,
$\{0\} \subsetneq W(x) \subsetneq V(x)$. 
\end{defn}

We recall the statement of the MET.

\begin{thm}[Multiplicative Ergodic Theorem for Invertible Matrix Cocycles]
\label{thm:met}
Let $(X,\mathcal{B},\mu,T)$ be an invertible and ergodic measure-preserving system, 
and let $A: \Z\times X \to GL_d(\R)$ be an invertible matrix cocycle. Suppose that $A$ 
satisfies 
\[
\int_X \log^+(\norm{A(1,x)})\ d\mu < \infty,
\quad \int_X \log^+(\norm{A(1,x)^{-1}})\ d\mu < \infty.
\]
Then there exist real numbers $\lambda_1 > \lambda_2 > \dots > \lambda_k > -\infty$, 
positive integers $m_1, \dots , m_k$ with $m_1 + \dots + m_k = d$, a subset $X_0$
with $\mu(X_0)=1$ and 
measurable families of subspaces $V_i(x)$ such that:
\begin{enumerate}
\item \textbf{Equivariance:}
For each $x\in X_0$, $A(1,x)V_i(x)=V_i(T(x))$ for $1\le i\le k$;
\item \textbf{Growth:}
For each $x\in X_0$, $1\le i\le k$ and each non-zero $v\in V_i(x)$, 
$\frac 1n\log\|A(n,x)v\|\to\lambda_i$ as $n\to\pm\infty$.
\end{enumerate}
\end{thm}

That is, the MET states that the bundle $X\times\R^d$ may be decomposed as a sum
of invariant bundles with different growth rates.

\begin{defn}[Block Diagonalization]
\label{defn:blockdiag}
Let $(X,\mathcal{B},\mu,T)$ be an ergodic dynamical system. We say that a 
measurable cocycle of $d$-by-$d$ matrices $A$ over $T$ can be put into 
\emph{block diagonal form with block sizes} $(m_1,\dots,m_k)$ over the field
$\F$ (or that $A$ is \emph{block diagonalizable}, if the block sizes are understood), 
if there exist positive integers $m_1,\dots,m_k$ with $k\geq 1$ and 
$m_1+\dots+m_k = d$, a $T$-invariant set of full measure $\tilde{X} \subset X$, 
and a measurable family of matrices $C: X\to GL_d(\F)$ such that 
$C(T(x))^{-1}A(1,x)C(x)$ is block diagonal, with block sizes $(m_1,\dots,m_k)$, 
for all $x\in \tilde{X}$. We also say that $A$ is \emph{cohomologous} to the 
resulting block diagonal matrix cocycle. 
\end{defn}

The previous definition may be extended to non-ergodic systems; in this case, 
the quantities $k$ and $m_1,\dots,m_k$ can be functions on $X$, but are now 
required to be constant on each ergodic component. We will only consider ergodic systems. 

The cocycle $A$ is block diagonalizable with block sizes $(m_1,\ldots,m_k)$
if and only if the bundle $X\times \mathbb F^d$ may be expressed as a sum of
$k$ vector bundles of dimensions $m_1,\ldots,m_k$. To see the equivalence,
notice that if $C(T(x))^{-1}A(1,x)C(x)$ is block diagonal, then 
the fibre, $V_i(x)$ of the $i$th bundle is the span of the $(m_1+\ldots+m_{i-1}+1)$st through
$(m_1+\ldots+m_i)$th columns of $C(x)$. Conversely, if $X\times\mathbb F^d$
is expressed as the sum of bundles $V_1$, $V_2$,\ldots, $V_k$ with fibres
of dimensions $m_1,\ldots,m_k$, then as in \cite{flq1}, we may measurably 
pick bases for each $V_i(x)$. Forming the matrix $C(x)$ as described above
yields the block diagonalization of the cocycle.

\begin{defn}[Block Triangularization]
\label{defn:blocktri}
We say that a measurable cocycle $A$ can be put into \emph{block triangular} 
form over the field $\F$, or is \emph{block triangularizable} over $\F$, 
if $A$ is block diagonalizable over $\F$ in the sense of Definition 
\ref{defn:blockdiag}, and each of the blocks in the decomposition is triangular. 
\end{defn}

\begin{rem}
\label{rem:upper-tri}
We generally assume that any triangularization is upper-triangularization; 
it is equivalent to consider lower triangularization, because if we have 
triangularizing matrices $C(x)$, we may instead use matrices $D(x) = C(x)F$, 
where $F$ is the matrix with ones from bottom-left to top-right and zeroes 
everywhere else. This obtains the opposite triangular form.
\end{rem}

\begin{rem}
\label{rem:2-by-2-tri}
In the case of a cocycle of $2\times 2$ matrices,
by the argument outlined below Definition \ref{defn:blockdiag}, we
see that reducibility of the bundle $X\times\mathbb F^2$
is equivalent to the triangularizability over $\mathbb F$ of the matrix cocycle.
\end{rem}

We may now state precisely the question that we wish to address: 
Given an invertible matrix cocycle, $A$, are the trivial bundles 
$X\times\R^2$ and $X\times\C^2$ irreducible under the action of $A$? 
In the matrix language, can $A$ necessarily be block triangularized?

The remainder of the paper proceeds as follows. In Section \ref{sec:suffcond}, 
we present the main theorem, Theorem \ref{thm:erg-tri}, which gives 
sufficient conditions for the trivial bundle for a $2$-by-$2$ orthogonal 
matrix cocycle to be irreducible over $\R$, 
to be irreducible over $\C$, and not to be cohomologous to a scalar multiple 
of the identity matrix. In Section \ref{sec:ceg}, we present three 
explicit examples illustrating the application of Theorem \ref{thm:erg-tri}, 
in both non-trivial situations. These examples provide a strong negative 
answer to the question posed above; namely, \emph{it is not always 
possible to block triangularize a matrix cocycle, even over the complex numbers}. 
In the language of bundles, \emph{there exists cocycles for which} $X\times\R^2$ 
\emph{and} $X\times\C^2$ \emph{are both irreducible.} 
In Section \ref{sec:erg-tri-proof}, we give the proof of Theorem 
\ref{thm:erg-tri}. Proofs of some of the more technical details are left to the Appendix.

\section{Sufficient conditions for irreducibility}
\label{sec:suffcond}

Denote the collection of real $2$-by-$2$ orthogonal matrices by $O_2(\R)$. 
Let $(X,\cB,\mu,T)$ be an invertible and ergodic measure-preserving system 
on a probability space, and let $A : \Z\times X\to O_2(\R)$ be a measurable 
matrix cocycle over $T$. For each $x\in X$, $A(1,x)$, as a map on $\R^2$, 
either rotates by some angle $\alpha_x$, or reflects in the line with 
angle $\beta_x$; let $X_r \subset X$ be where $A(1,x)$ is a rotation, 
and let $X_f \subset X$ be where $A(1,x)$ is a reflection.

We choose to restrict our study here to orthogonal matrices, because 
in the dimension $2$ case, the MET guarantees that if a cocycle has 
two different Lyapunov exponents, the cocycle is diagonalizable. 
Orthogonal matrices do not change the norm of any vectors, and so 
the only Lyapunov exponent for cocycles of orthogonal matrices is $0$. 
Therefore, the MET yields the trivial decomposition for the cocycle.

Denote $\T = \R/\Z$, and $\Z_2 = \Z/2\Z$. For each $x\in X$, define 
the maps $f_x: \T\to\T$ and $g_x: \Z_2\to\Z_2$ by: 
\[
f_x(y) = \begin{cases} y + \dfrac{\alpha_x}{\pi} & x\in X_r, \\
\dfrac{2\beta_x}{\pi} - y & x\in X_f, \end{cases}\ \qquad 
g_x(a) = \begin{cases} a & x\in X_r, \\ a+1 & x\in X_f. \end{cases}
\] 
From these maps, define skew products $R : X\times\Z_2\to X\times\Z_2$ 
given by $R(x,a) = (T(x),g_x(a))$, and $S : X\times\T \to X\times\T$ 
given by $S(x,y) = (T(x),f_x(y))$. $R$ and $S$ are measure-preserving 
transformations on $X\times\Z_2$ and $X\times\T$, when each space is 
equipped with Haar measure, $\mu\times c$ and $\mu\times\lambda$, 
respectively ($c$ is normalized counting measure on $\Z_2$ and 
$\lambda$ is Lebesgue measure).

\begin{thm}
\label{thm:erg-tri}
Let $(X,\cB,\mu,T)$ be an invertible ergodic measure-preserving system 
over a probability space. Let $A : \Z\times X \to O_2(\R)$ be a measurable 
cocycle over the map $T$. Let $R$ and $S$ be as constructed above.
\begin{enumerate}
\item If $S$ is ergodic, $X\times\R^2$ is irreducible under the action of 
$A$.
\item If both $R$ and $S$ are ergodic, then $X\times\C^2$ is irreducible under the
action of $A$.
\item If at least one of $R$ or $S$ is ergodic, then the cocycle $A$ is not 
cohomologous to a cocycle of the form $\lambda(x)I$, {i.e.} a scalar multiple of the identity.
\end{enumerate}  
\end{thm}

\begin{rem}
\label{rem:tri}
Notice that in the case of $O_2(\R)$-valued cocycles, if $V=\{(x,V(x))\}$ is a real
invariant sub-bundle, then so is $V^\perp=\{(x,V^\perp(x))\}$, where $V^\perp(x)$ is 
just the orthogonal complement of $V(x)$; the same holds for the Hermitian complement, in the
case that $V$ is a complex vector bundle. Hence, for $O_2$-valued cocycles, if $X\times\F^2$ 
is a reducible bundle, then it may be decomposed as a sum of line bundles. 

In matrix language, if an $O_2(\R)$ cocycle can be block triangularized, then it can be
block diagonalized.
\end{rem}

\begin{rem}
\label{rem:cases}
The third part of Theorem \ref{thm:erg-tri}
is the cocycle analogue of the linear algebra fact that 
given two eigenvectors for the same eigenvalue of a $2$-by-$2$ matrix, 
either they are scalar multiples of each other, or the matrix is similar 
to a scalar multiple of the identity. If the cocycle is cohomologous
to a scalar cocycle, then there is a continuum of proper sub-bundles of $X\times\mathbb F^2$. 
This is ruled out by the ergodicity of either $R$ or $S$.
\end{rem}

\begin{rem}
\label{rem:necessary}
Theorem \ref{thm:erg-tri} gives \emph{sufficient} conditions for $X\times\F^2$
to be an irreducible bundle for the cocycle $A$, and for $A$ to be non-cohomologous 
to a scalar multiple of the identity. None of these conditions are necessary; 
counter-examples are given in the Appendix.
\end{rem}

In the next section, we use Theorem \ref{thm:erg-tri} to show the existence of
matrix cocycles where the trivial bundle is irreducible. The proof of Theorem 
\ref{thm:erg-tri} will be presented in Section \ref{sec:erg-tri-proof}.

\section{Examples of irreducibility}
\label{sec:ceg}

In this section, we discuss three examples; one is a cocycle where the trivial bundle
$X\times \C^2$ is reducible, but $X\times\R^2$ is irreducible, and the other two 
are cases where
$X\times\C^2$ is irreducible (and hence so is $X\times\R^2$).

\subsection{Example 1: rotation cocycle over a rotation}
\label{subsect:rot-cocycle-over-rot}

Let $(\T,\cB,\lambda,T)$ be the irrational rotation by $\eta$ over the 
unit interval with normalized Lebesgue measure, and consider the matrix 
cocycle $A$ generated by 
\[
A(1,x) = \begin{bmatrix} \cos(\pi{x}) & -\sin(\pi{x}) \\
\sin(\pi{x}) & \cos(\pi{x}) \end{bmatrix} = \mathrm{rot}_{\pi{x}}.
\] 

\begin{prop}
\label{prop:rot-cocycle-over-rot}
The bundle $X\times\C^2$ is reducible under the action of $A$, but
$X\times \R^2$ is irreducible under the action of $A$.
\end{prop}

\begin{proof}
To see that $X\times\R^2$ is irreducible, we first compute the map $S$. 
Each matrix $A(1,x)$ is a rotation in $\R^2$ by $\alpha_x = \pi{x}$, 
which allows us to compute the map $f_x$ as outlined above. We 
obtain $f_x(y) = y+x$, and then we compute the map $S$ on $\T^2$ to be 
$S(x,y) = (x+\eta,y+x)$. It is well known (see, for example, 
\cite{furstenberg}, Theorem 2.1) that $S$ is an ergodic map with 
respect to Lebesgue measure on $\T^2$. Applying Theorem 
\ref{thm:erg-tri} yields that $X\times\R^2$ is irreducible.

On the other hand, observe that $A(1,x)\begin{bmatrix}1\\i\end{bmatrix}=e^{-\pi i x}
\begin{bmatrix}1\\i\end{bmatrix}$ and $A(1,x)\begin{bmatrix}1\\-i\end{bmatrix}=e^{\pi i x}
\begin{bmatrix}1\\-i\end{bmatrix}$
Hence $X\times\C^2$ is reducible, as 
\[\left(X\times\subspan{\begin{bmatrix}1\\ i\end{bmatrix}}\right)\oplus \left(X\times \subspan{\begin{bmatrix}1\\-i\end{bmatrix}}\right). \]
\end{proof}

\subsection{Example 2: rotation and flip cocycle over a rotation}
\label{subsect:flip-cocycle-over-rot}

Let the base dynamics space be the same as in the previous example: 
$(\T,\cB,\lambda,T)$. This time, for an irrational number 
$\alpha$ in $(0,1)$, we define a matrix cocycle $A$ over 
$T$, with 
\[
A(1,x) = \begin{cases} \ \begin{bmatrix}
\cos(\pi\alpha) & -\sin(\pi\alpha) \\
\sin(\pi\alpha) & \cos(\pi\alpha)
\end{bmatrix} & x \in [0,1-\eta), \\[20pt]
\ \begin{bmatrix}
1 & 0 \\
0 & -1
\end{bmatrix} & x \in [1-\eta,1).
\end{cases}
\] 

\begin{prop}
\label{prop:flip-cocycle-over-rot}
The trivial bundle $X\times\C^2$ is irreducible for the cocycle $A$ over 
$(\T,\cB,\lambda,T)$ as defined above.
\end{prop}

We leave the proof to the Appendix. The proof uses a result by Schmidt 
stemming from his work on ergodic transformation groups \cite{ks_cocycle}. 
We note that we need to have $A$ take values outside of the rotations $SO_2(\R)$, 
since the same proof as in the previous example shows that $X\times\C^2$ is 
reducible under the action of \emph{any} measurable cocycle with values in $SO_2(\R)$.

\subsection{Example 3: rotation and flip cocycle over a Bernoulli shift}
\label{subsect:flip-cocycle-over-shift}

In the next example, the base map will be a Bernoulli shift instead of a
circle rotation. We again obtain irreducibility of the bundle, indicating
that reducibility of the bundle is not primarily determined by the properties
of the base system.

Let $(X,\cB,\mu,T)$ denote the left Bernoulli shift on two symbols, each 
with weight $\frac{1}{2}$. Let $A$ be a matrix cocycle over $T$, 
generated by the map $A(1,\cdot) : X\to O_2(\R)$, given by: 
\[
A(1,x) = \begin{cases} \
\begin{bmatrix}
\cos(\pi\alpha) & -\sin(\pi\alpha) \\
\sin(\pi\alpha) & \cos(\pi\alpha)
\end{bmatrix} & x_0 = 0, \\[20pt]
\ \begin{bmatrix}
1 & 0 \\
0 & -1
\end{bmatrix} & x_0 = 1.
\end{cases}
\] 

\begin{prop}
\label{prop:flip-cocycle-over-shift}
The trivial bundle $X\times\C^2$ is irreducible for the cocycle 
$A$ over $(X,\cB,\mu,T)$ as defined above.
\end{prop}

Again, we leave the proof to the Appendix. This example has the feature 
that the proof is self-contained, relying only on Fourier analysis.

\section{Proof of Theorem \ref{thm:erg-tri}}
\label{sec:erg-tri-proof}

\begin{proof}
We begin with the proof of irreducibility of the bundle $X\times\C^2$.
The proof will proceed in three steps.

\emph{Step 1:} Let $\mathrm{Gr}_1(\C^2)$ denote the complex 
Grassmannian of $1$-dimensional subspaces of $\C^2$. It is well-known 
that $\mathrm{Gr}_1(\C^2)$ is homeomorphic to $\bar{\C}$, the one-point 
compactification of $\C$. In particular, if we choose 
\[
v_1 = \begin{bmatrix} 1 \\ i \end{bmatrix},\quad 
v_2 = \begin{bmatrix} 1 \\ -i \end{bmatrix},
\] 
then a homeomorphism is given by $z \leftrightarrow \subspan\{v_1+zv_2\}$ 
for $z\in\C$, and $\infty \leftrightarrow \subspan\{v_2\}$; this yields 
coordinates for the Grassmannian.

Consider the skew product of $T$ and the generator of $A$, $A(1,\cdot)$, 
on the space $X\times \mathrm{Gr}_1(\C^2)$ (as an invertible $2$-by-$2$ matrix, 
$A(1,x)$ acts naturally on $\mathrm{Gr}_1(\C^2)$). In the coordinates 
above, we obtain a map $N$ on $X\times\bar{\C}$ as follows: 
for $z\in\C\setminus\{0\}$, 
\[
N(x,z) = \begin{cases} (T(x), e^{2i\alpha_x}z) & x\in X_r, \\
(T(x),{e^{4i\beta_x}}/{z}) & x\in X_f, \end{cases}
\] 
whereas for $z=0$ or $z=\infty$, we have 
\begin{gather*} 
N(x,0) = \begin{cases} (T(x),0) & x\in X_r, \\
 (T(x),\infty) & x\in X_f, \end{cases} \\
N(x,\infty) = \begin{cases} (T(x),\infty) & x\in X_r, \\ (T(x),0) & x\in X_r. \end{cases}
\end{gather*}
The form of $N$ stems from the fact that $v_1$ and $v_2$ are eigenvectors 
for the rotation matrices, and are swapped and scaled by the reflection matrices.

To complete the setup, denote the complex unit circle by $\mathbb{S}$, 
let $\iota : \bar{\C}\setminus \mathbb{S} \to \Z_2$ be given by 
\[
\iota(z) = \begin{cases} 0 & \abs{z} < 1, \\ 1 & \abs{z} > 1, \end{cases}
\]
and let $\tau : \bar{\C}\setminus\{0,\infty\} \to \T$ be given by 
$\tau(z) =\arg(z)/(2\pi)$. These maps are measurable, and will 
be used to relate $N$ with $R$ and $S$.

\vskip \baselineskip

\emph{Step 2:} Assume, for the sake of contradiction, that $V$ is an
invariant one-dimensional sub-bundle  of $X\times\C^2$.
Denote the coordinates of $V(x)$ in $\bar{\C}$ by $w(x)$. 
The following lemma sets the stage for the remainder of the proof.

\begin{lem}
\label{lem:inv-graph}
Let $N$ be the map described in Step 1, and let $w(x)$ be described 
as above. Let $P_C$ be the union of two circles $\{\abs{z} = C\} \cup 
\{\abs{z} = \frac{1}{C}\},$ where if $C=0$ then $\frac{1}{C} = \infty$, 
and $P_1 = \mathbb{S}$. Then the graph of $w$ is invariant under $N$, 
and is contained in $X\times P_C$ for some $C\in [0,1]$. Moreover, 
$X\times P_C$ is an invariant set under $N$, for any $C\in[0,1]$.
\end{lem}

\begin{proof}
The invariance of $V$ and the definition of the map $N$ provide 
almost-everywhere invariance of the graph of $w$ under $N$. 
For the containment of the graph inside $X\times P_C$ for some $C$, 
let $k(x) = \min\{\abs{w(x)},\frac{1}{\abs{w(x)}}\}$, and observe that 
$k$ is an invariant measurable function with respect to the ergodic map $T$, 
hence is almost-everywhere constant with respect to $\mu$, with 
value $C\in [0,1]$. Then the graph of $w$ is contained in $X\times P_C$, 
up to a set of measure zero. The set $X\times P_C$ is easily shown to be 
almost-everywhere invariant by applying $N$ to a point in the set. 
Finally, we may remove an invariant measure-zero subset of $X$ 
on which equivariance of $w$ fails, and separately where containment 
inside $X\times P_C$ fails. This completes the proof.
\end{proof}

\vskip \baselineskip

Step 3: By the above lemma, $X\times \mathbb{S}$ and its complement 
are invariant sets for $N$, and the graph of $w$ lies entirely in one 
of them. We split the remainder of the proof into two cases. First, 
consider the case where the graph of $w$ lies in $X\times \mathbb{S}$. 
Using the definition of the map $\tau$ from Step 1, an explicit 
calculation shows that 
\[(\id\times\tau)\circ \restr{N}{X\times\mathbb{S}} = 
S \circ (\id\times\tau).
\] 
Thus the graph of $\tau \circ w$ (a subset of $X\times \T$) is invariant 
under $S$. Let $h(x,y) = \abs{y-\tau\circ w(x)}_{\T}$ be the distance 
on the torus $\T$. By using the explicit form of $S$ (as before Theorem 
\ref{thm:erg-tri}), $h$ is shown to be an $S$-invariant function, 
and it is clearly measurable. Since $h$ is non-constant, this contradicts 
the ergodicity of $S$.

Then, consider the case that the graph of $w$ lies in the complement of 
$X\times\mathbb{S}$. Using the map $\iota$, we get that 
\[
(\id\times\iota) \circ \restr{N}{X\times\{\C\setminus\mathbb{S}\}} 
= R \circ (\id\times\iota),
\] 
so analogously to the previous case, the graph of $\iota\circ w$ is 
invariant under $R$. We apply Fubini's theorem to show that the graph 
of $\iota\circ w$ has measure $\frac{1}{2}$, so since the graph is 
$R$-invariant, this contradicts the ergodicity of $R$.

The two cases together show that $w$ cannot exist, and thus the one-dimensional 
bundle, $V$, cannot exist as assumed; therefore $X\times\C^2$ is irreducible.

Next, we show that if $S$ is ergodic, then $X\times\R^2$ is irreducible;
this is part (2) of the theorem. Instead of the complex Grassmannian, we consider 
the real Grassmannian, $\mathrm{Gr}_1(\R^2)$. It is easy to show that it is 
homeomorphic to the unit circle $\T$ equipped with the usual topology, so we have 
coordinates for $\mathrm{Gr}_1(\R^2)$. A similar argument to that in Step 1 for the 
complex case computes the skew product of $T$ and $A(1,\cdot)$ over $X\times\T$, 
which is exactly the map $S$. The exact same arguments in the remaining steps 
applied to the map $S$ go through, and hence $X\times\R^2$ is irreducible.

Finally, we prove part (3) of the theorem. Suppose, on the contrary, that $A$ 
is cohomologous to a scalar multiple of the identity, so that there exists a 
measurable matrix function $C(x)$ and a measurable function $\lambda : X \to \C$ 
such that on a set of full measure in $X$,
\[
C(T(x))^{-1}A(1,x)C(x) = \lambda(x)I.
\]
In the previous parts of the proof, we found that if there are equivariant 
subspaces which do not lie on the unit circle of the Grassmannian 
(in $\bar{\C}$ coordinates), then the map $R$ cannot be ergodic, 
and if there are equivariant subspaces which do not lie at the poles 
($0$ or $\infty$ in $\bar{\C}$), then the map $S$ cannot be ergodic. 
We will use the functions $v_1(x)$ and $v_2(x)$ to construct an 
equivariant family of subspaces which lie neither at the poles, 
nor on the unit circle. This will imply that neither $R$ nor $S$ is 
ergodic, which is the contrapositive of statement (3) in the theorem.
Notice that for fixed $a$ and $b$ (not both 0), $\{(x,\subspan\{av_1(x)+bv_2(x)\})\colon
x\in X\}$ is a one-dimensional invariant bundle; if we can find $a$ and $b$ 
such that the subspaces in this bundle live away from the unit circle or the poles, 
then we will be done.

Let $q$ be the homeomorphism from $\mathrm{Gr}_1(\C^2)$ to $\bar{\C}$, 
and let $h : (\C\times\C\setminus\{(0,0)\})\times X \to [0,1]$ be given by 
\[
h(a,b,x) = \min_{j=\pm 1}\{\abs{q(\subspan\{av_1(x)+bv_2(x)\})}^{j}\}.
\] 
This function is measurable and non-negative, with values in $[0,1]$
(similar to the function $k$ in Lemma \ref{lem:inv-graph}). For fixed $x\in X$, 
$(a,b) \mapsto h(a,b,x)$ is continuous, and for fixed $a,b \in \C$ 
not both zero, the function $x \mapsto h(a,b,x)$ is measurable, and 
$T$-invariant. The latter is shown by an explicit computation, 
utilizing the relations between $A$ and $v_i(x)$, as well as the 
induced action of $A$ on the coordinates of the Grassmannian, 
similar to the proof of Lemma \ref{lem:inv-graph}.

Hence, since $T$ is ergodic with respect to $\mu$,  we see that
for any fixed $a,b\in \C$, 
$x \mapsto h(a,b,x)$ is almost everywhere constant. Note that if 
$h(a,b,x) = 0$, then the span of $av_1(x)+bv_2(x)$ is at one of 
the poles ($0$ or $\infty$), and if $h(a,b,x) = 1$, then the 
span of $av_1(x)+bv_2(x)$ is on the unit circle. We will be done 
if we can find $a,b \in \C$ such that the almost everywhere value of 
$h(a,b,x)$ is neither $0$ nor $1$; then the map 
\[ x \mapsto \subspan_{\C}\{av_1(x)+bv_2(x)\} \]
is exactly the equivariant family of subspaces we desire. 

To do this, let $H : \C\times\C\setminus \{(0,0)\} \to [0,1]$ be given by 
\[
H(a,b) = \int_X h(a,b,x)\ d\mu(x);
\]
then $H(a,b)$ is exactly the almost everywhere value for 
$x \mapsto h(a,b,x)$. We will show that $H$ takes a value which is 
neither $0$ nor $1$. For a fixed $(a,b)$, let $X_0 \subset X$ be a set of 
full measure such that for all $x\in X_0$, $h(a,b,x) = H(a,b)$. Let 
$(a_n,b_n) \to (a,b)$, and for each $n\in\N$, find $X_n \subset X$ 
such that $X_n$ has full measure, and for all $x\in X_n$, we have 
$h(a_n,b_n,x) = H(a_n,b_n)$. Let 
\[
\tilde{X} = \bigcap_{n=0}^\infty X_n.
\]
Then $\mu(\tilde{X}) = 1$, and for any $x \in \tilde{X}$, we have: 
\[
H(a,b) = h(a,b,x) = \lim_{n\to\infty} h(a_n,b_n,x) = \lim_{n\to\infty} H(a_n,b_n),
\] 
using the continuity of $h$ for fixed $x$. Hence $H$ is continuous.

Let $Q$ be a countable dense subset of $\C\times\C$, which does not 
include $(0,0)$. For each $(r,s)\in Q$, find $Y_{r,s} \subset X$ such that 
$\mu(Y_{r,s}) = 1$ and for all $x\in Y_{r,s}$, we have $h(r,s,x) = H(r,s)$. 
Let 
\[ Y = \bigcap_{(r,s)\in Q} Y_{r,s}. \]
Then $\mu(Y) = 1$. Let $x\in Y$. For any $(a,b) \ne (0,0)$, find 
$(r_n,s_n) \in Q$ converging to $(a,b)$, we get: 
\[
H(a,b) = \lim_{n\to\infty} H(r_n,s_n) = \lim_{n\to\infty} h(r_n,s_n,x) = h(a,b,x),
\] 
again using the continuity of $h$ for fixed $x$, as well as the continuity of $H$.

Finally, note that by definition of $v_1(x)$ and $v_2(x)$, there exists a set 
$G$ of full measure in $X$ such that $\{v_1(x),v_2(x)\}$ is a basis for $\C^2$. 
So for any $x\in G$, the map $(a,b) \mapsto h(a,b,x)$ is surjective, and hence 
takes values which are neither $0$ nor $1$. Now, $G\cap Y$ has full measure, so 
find $x\in G\cap Y$, and find $(a,b)$ such that $h(a,b,x) \notin \{0,1\}$. 
Since $x$ is also in $Y$, we have $H(a,b) = h(a,b,x) \notin \{0,1\}$, and we have 
completed the proof. 
\end{proof}

\appendix

\section{Other proofs}

\begin{proof}[Proof of Proposition \ref{prop:flip-cocycle-over-rot}]
$A(1,x)$ is a rotation by $\pi\alpha$ for $x\in[0,1-\eta)$, and a reflection 
in the horizontal axis for $x\in[1-\eta,1)$. In our notation, 
$X_r = [0,1-\eta)$, with a fixed rotation angle $\alpha_x = \pi{\alpha}$, 
and $X_f = [1-\eta,1)$, with fixed reflection axis $\beta_x = 0$. 
Substituting this into our maps $S$ and $R$ gives us: 
\begin{gather*}S(x,y) = 
\begin{cases} (x+\eta,y+\alpha) & x\in [0,1-\eta), \\
(x+\eta,1-y) & x\in [1-\eta,1); \end{cases} \\
R(x,a) = 
\begin{cases} (x+\eta,a) & x\in [0,1-\eta), \\ 
(x+\eta,a+1) & x\in [1-\eta,1). \end{cases}
\end{gather*} 
It is then sufficient to show that these maps are ergodic, as afterwards 
we simply apply Theorem \ref{thm:erg-tri}.

The first step for both of these claims will be to induce the map on a 
subset of positive measure; ergodicity of the induced map will then 
imply ergodicity of the original map. We deal first with $R$, then with $S$.

In the case of $R$ on $\T\times\Z_2$, we induce on the set 
$B = [1-\eta,1)\times\Z_2$. Since the set on which we are inducing 
is the product of $[1-\eta,1)$ and the whole space $\Z_2$, the return 
time for $R$ only depends on the map $T$ and the first coordinate $x$. 
The base map for this skew product is an irrational rotation, and 
inducing an irrational rotation on an interval of the same size 
as the rotation angle yields a map which is isomorphic to a 
different irrational rotation (see \cite{rauzy_iet}, noting that 
rotations are special cases of interval exchange transformations), 
with the rotation angle given by the fractional part of $\frac{1}{\eta}$, 
denoted $\beta$. Using this and using the fact that the only action in the 
$\Z_2$ component is when the map $R$ first leaves the set $B$, we obtain 
that the induced map of $R$ on $B$ is isomorphic to the map 
$R_B : \T\times\Z_2 \to \T\times\Z_2$, given by 
\[
R_B(x,a) = (x+\beta, a+1) = (x,a)+(\beta,1).
\]
Another well-known result (Theorem 1.9 in \cite{walters}) says that because 
$\T\times\Z_2$ is compact, $R_B$ is a rotation, and $\{(\beta,1)\}^n$ is 
dense in $\T\times\Z_2$, $R_B$ is ergodic; we are done in the case of $R$.

In the case of $S$ on $\T\times\T = \T^2$, we must do a bit more work. 
Similarly to before, we induce on the set $B=[1-\eta,1)\times\T$, and 
obtain a new map $S_B$ on $\T^2$, which is given by 
\[
S_B(x,y) = \begin{cases} (x+\beta,k\alpha-y) & x\in[1-\beta,1), \\
(x+\beta,(k-1)\alpha-y) & x\in[0,1-\beta). \end{cases}
\]
If we were to apply $S_B$ twice in a row, we would eliminate the flip 
in the $y$ coordinate. This idea is justified, because it is easy 
to show that if a map $T$ is measure-preserving and $T^2$ is ergodic, 
then $T$ is also ergodic.

Hence, we square $S_B$ to obtain the map $P = S_B^2$, given by \[P(x,y) = \begin{cases}
(x+2\beta,y) & x\in[0,1-2\beta), \\
(x+2\beta,y+\alpha) & x\in[1-2\beta,1-\beta), \\
(x+2\beta,y-\alpha) & x\in[1-\beta,1). \\
\end{cases}\] We again induce, this time on the set $[1-2\beta,1)\times\T$, 
and after another coordinate change, we get a map $Q$ acting on $\T^2$. 
Setting $\zeta$ to be the fractional part of $\frac{1}{2\beta}$, 
we see that $Q$ is given by 
\[
Q(x,y) = \begin{cases} (x+\zeta,y-\alpha) & x\in[0,\frac{1}{2}), \\
(x+\zeta,y+\alpha) & x\in[\frac{1}{2},1). \end{cases}
\]

We now appeal to a result from the literature: 

\begin{prop}[Schmidt, \cite{ks_cocycle}]
\label{prop:ks-res}
Consider the space $\T\times \alpha\Z$ as defined in the previous proposition. 
Define the map $\tilde{Q}$ on $\T\times\alpha\Z$ by 
\[
\tilde{Q}(x,n\alpha) = \begin{cases}
(x+\zeta,(n+1)\alpha) & x\in [0,\frac{1}{2}), \\
(x+\zeta,(n-1)\alpha) & x\in [\frac{1}{2},1).
\end{cases}
\]
Then $\tilde{Q}$ is an ergodic measure-preserving transformation.
\end{prop}

To use this proposition, we translate it from the $\sigma$-finite case 
to the finite case by this proposition: 

\begin{prop}
\label{prop:sig-finite}
Suppose $\sigma:\T\to \T$ is measure-preserving and ergodic with respect 
to Lebesgue measure, and let $f:\T\to \R$ be a measurable function, with 
range $f(\T) \subset \alpha\Z$, where $\alpha$ is irrational. Let $\T^2$ 
have the usual Lebesgue product measure and Borel sets, and let 
$T_f : \T^2 \to \T^2$ be the skew product extension of $\sigma$ and 
$f$ to $\T^2$, so that: 
\[
T_f(x,y) = (\sigma(x), y + f(x)).
\] 
Let $\tilde{T}_f:\T\times\alpha\Z \to \T\times\alpha\Z$ be the skew 
product extension of $\sigma$ and $f$ to $\T\times\alpha\Z$ with the 
($\sigma$-finite) product measure $\lambda\times c$ (Lebesgue and counting, with the 
discrete $\sigma$-algebra for the counting measure), so that: 
\[
\tilde{T}_f(x,n\alpha) = (\sigma(x), n\alpha + f(x)).
\]
Then if $\tilde{T}_f$ is ergodic, so is $T_f$.
\end{prop}

\begin{proof}
Let $h:\T^2\to \R$ be a bounded measurable function invariant under 
$T_f$, so $h\circ T_f = h$. We shall show that $h$ must be a.e. constant; 
this will imply that $T_f$ is ergodic. For $y\in \T$, define the 
measurable map 
\[
\pi_y: \T\times\alpha\Z \to \T^2,\ \pi_y(x,n\alpha) = (x,y + n\alpha).
\] 
Then we see that $T_f\circ\pi_y = \pi_y\circ\tilde{T}_f$. In addition, 
define $\tilde{h}_y = h\circ\pi_y$, so that $\tilde{h}_y$ is a measurable 
function defined on $\T\times \alpha\Z$. Since $\pi_y$ intertwines the 
dynamics on the two spaces, we get the following:
\[
\tilde{h}_y\circ\tilde{T}_y = h\circ\pi_y\circ\tilde{T}_f = 
h\circ T_f\circ\pi_y = h\circ\pi_y = \tilde{h}_y.
\]
Thus $\tilde{h}_y$ is invariant under $\tilde{T}_f$, and so is constant 
a.e. with respect to the product measure $\lambda\times c$, since $\tilde{T}_f$ is ergodic.

We wish to use the fact that $\tilde{h}_y$ is a.e. constant for each 
$y\in\T$ to show that $h$ is constant a.e. To do this, we make an intermediate 
step. Define 
\[
I : \T \to \R,\ I(y) = \int_{0}^{1}{h(x,y)\ dx} = \int_0^1 \tilde{h}_y(x,0)\ dx.
\] 
Because $h$ is bounded, $I$ is finite, and by Fubini's theorem, $I$ is 
measurable. Moreover, we have the following, since $\tilde{h}_y$ is 
a.e. constant on $\T\times\alpha\Z$:
\begin{align*}
I(y+\alpha) & = \int_{0}^{1}{h(x,y+\alpha)}\ dx = 
\int_0^1{\tilde{h}_y(x,\alpha)}\ dx = \int_0^1{\tilde{h}_y(x,0)}\ dx \\
& = \int_{0}^{1}{h(x,y)}\ dx = I(y).
\end{align*}
The map $y \mapsto y+\alpha$ is ergodic on $\T$, thus we see that $I$ is a.e. 
constant on $\T$; write $I(y) = C$ for a.e. $y\in\T$. Note that for all $y$, 
$\tilde{h}_y$ is a.e. constant on $\T\times\alpha\Z$, so we see that for a.e. 
$x\in\T$, $\tilde{h}_y(x,0) = I(y)$. Denote $Y_G = \set{y\in\T}{I(y) = C}$; 
this set has full measure in $\T$. If $y\in Y_G$, then for a.e. $x$, 
$h(x,y) = \tilde{h}_y(x,0) = C$. Computing the measure of the set of 
points where $h\ne C$ via Fubini's Theorem yields the final statement: 
$h = C$ almost everywhere. Hence $T_f$ is ergodic.
\end{proof}

$Q$ and $\tilde{Q}$ are related in exactly the correct way for Proposition 
\ref{prop:sig-finite} to be applied, and so $Q$ is ergodic. This proves 
that the original map $S$ is ergodic, and so we are finished.
\end{proof}

\begin{proof}[Proof of Proposition \ref{prop:flip-cocycle-over-shift}]
The proof that $X\times\C^2$ is irreducible is similar, in outline, to the proof 
that the cocycle in Example 2 is not triangularizable. We give only a brief 
sketch and leave the details to \cite{horan_thesis}.

The map $R$, on $X\times\Z_2$, is shown to be ergodic by inducing on the product 
of $\Z_2$ with the set of binary strings with $1$ as the $0$-th 
component, and obtaining the product of a Bernoulli shift on countably 
many symbols ($\N$), which is strongly mixing, and the translation by $1$ on 
$\Z_2$, which is ergodic. Their product is therefore ergodic, and so the map 
$R$ is ergodic.

The map $S$, on $X\times\T$, requires a similar process to that in Example 2. 
After inducing, squaring, and inducing again, the resulting map is a skew 
product over a Bernoulli shift on strings with symbols from $\N\times\N$, 
where the action on the torus $\T$ is a rotation determined by the 
$0$-th component of the string. One then applies Fourier 
analysis techniques to show that this new map is strongly mixing, 
and hence the map $S$ is ergodic.
\end{proof}

\section{Counter-examples to necessity of conditions in Theorem \ref{thm:erg-tri}}

As mentioned in Remark \ref{rem:necessary}, none of the sufficient conditions
listed in Theorem \ref{thm:erg-tri} are necessary. We will illustrate with some simple counterexamples.

For condition (1), consider the system $(\T,\cB,\lambda,T)$, where $\T$ is the 
$1$-torus $\R/\Z$ equipped with the usual Borel $\sigma$-algebra and 
normalized Lebesgue measure and $T(x) = x+\eta$ is an irrational 
rotation by $\eta \in [0,1)\setminus\Q$, and let 
$A(1,x) = \mathrm{rot}\left(\frac{\pi}{3}\right).$ In this case, the map 
$S : \T\times\T \to \T\times\T$ is given by \(S(x,y) = \left(x+\eta,y+\frac{1}{3}\right)\).
It is easy to see that this map is not ergodic: the set 
\(B = \T\times \left( \left(0,\frac{1}{6}\right) \cup \left(\frac{1}{3},
\frac{1}{2}\right) \cup \left(\frac{2}{3},\frac{5}{6}\right) \right) \) is 
$S$-invariant, and has measure $\frac{1}{2}$. However, the cocycle $A$ is irreducible over $\R$.
To see this, consider the projection $\id\times\pi_3$, where $\pi_3 : \T\to \Z_3$ is given by 
$\pi_3(y) = \lfloor 3y \rfloor$. This projection induces is a factor map 
$\tilde{S} : \T\times\Z_3 \to T\times\Z_3$ for $S$, given by 
$\tilde{S}(x,a) = (x+\eta,a+1)$. This map \emph{is} ergodic, which prevents the existence of a 
real equivariant family of subspaces for $A$, because the resulting invariant graph for $S$
factors to an invariant graph for $\tilde{S}$, and this contradicts the fact that $\tilde{S}$ 
is ergodic.

The same example can be used to show that condition (3) is also not 
necessary; the associated map $R : \T\times\Z_2 \to \T\times\Z_2$ is just $R(x,a) = (x+\eta,a)$,
which is also not ergodic. However, the cocycle is not cohomologous to a
scalar multiple of the identity. Suppose it were; following the proof of condition (3) 
in Theorem \ref{thm:erg-tri}, we obtain an equivariant family of subspaces which yields an invariant
graph for $S$. This yields an invariant graph for $\tilde{S}$, which 
contradicts the ergodicity of $\tilde{S}$.

For condition (2), modify the cocycle above as in Section \ref{subsect:flip-cocycle-over-shift}.
Let $(\Omega,\mu,\sigma)$ be the left Bernoulli shift on $\{0,1\}^{\Z}$, where $0$ and $1$ are both given weight $\frac{1}{2}$, and let $A$ be the cocycle over $\sigma$ generated by 
\[ 
A(1,\omega) = \begin{cases} \ \begin{bmatrix}
\cos(\frac{\pi}{3}) & -\sin(\frac{\pi}{3}) \\
\sin(\frac{\pi}{3}) & \cos(\frac{\pi}{3})
\end{bmatrix} & \omega_0 = 0, \\[20pt]
\ \begin{bmatrix}
1 & 0 \\
0 & -1
\end{bmatrix} & \omega_0 = 1.
\end{cases}
\] 
In this case, the map $R$ is still ergodic, but $S$ is not ergodic, because 
$B = \T\times \left( \left(\frac{1}{9},\frac{2}{9}\right) \cup \left(\frac{4}{9},
\frac{5}{9}\right) \cup \left(\frac{7}{9},\frac{8}{9}\right) \right)$ is $S$-invariant,
with measure $\frac{1}{3}$. Once again, using a similar construction to the above (taking care of the points 
$0,\frac{1}{3},\frac{2}{3}$ separately), we obtain a similar factor map $\tilde{S}$ of $S$, which is ergodic 
(by a similar proof to that of Proposition 
\ref{prop:flip-cocycle-over-shift}). Applying the arguments used in the proof of (2) in Theorem
\ref{thm:erg-tri}, but to $R$ and $\tilde{S}$ instead of $R$ and $S$, we see that $A$ has no equivariant
family of one-dimensional subspaces over $\C$. Therefore, condition (2) is not necessary.

\bibliographystyle{abbrv}
\bibliography{ergbib}

\end{document}